\documentclass[12pt]{amsart}
\usepackage{amsmath,amssymb,amsthm,xspace,a4wide,amscd}
\usepackage[all]{xy}
\usepackage{euscript}
\newtheorem{theorem}{Theorem}[section]

\newtheorem{remark}[theorem]{Remark}
\newtheorem{lemma}[theorem]{Lemma}
\newtheorem{proposition}[theorem]{Proposition}
\newtheorem{corollary}[theorem]{Corollary}

\usepackage{xcolor}

\title[Whittaker modules]{Affine Lie algebras representations induced from Whittaker modules}
\author[M.C.Cardoso]{Maria Clara Cardoso}
	\address{Instituto de Matem\'atica e Estat\'istica \\ Universidade de S\~ao Paulo \\ Sao Paulo \\ Brazil}
	\email{mariaclaracardoso23@gmail.com}
\author[V. Futorny]{Vyacheslav Futorny}
	\address{Instituto de Matem\'atica e Estat\'istica \\ Universidade de S\~ao Paulo \\ Sao Paulo \\ Brazil, and  International Center for Mathematics, SUSTech, Shenzhen, China}
	\email{futorny@ime.usp.br}

\date{\today}

\begin{document}

\begin{abstract}
We use induction from parabolic subalgebras with infinite-dimensional Levi factor to construct new families of irreducible representations for arbitrary Affine Kac-Moody algebra. Our first construction defines a functor from the category of Whittaker modules over the Levi factor of a parabolic subalgebra to the category of modules over the Affine Lie algebra. The second functor sends tensor products of a module over the affine part of  
 the Levi factor (in particular any weight module) and of a Whittaker module over the complement Heisenberg subalgebra to the Affine Lie algebra
modules. Both functors preserves irreducibility when the central charge is nonzero.
\medskip

\medskip

\medskip

\end{abstract}

\maketitle
\thispagestyle{empty}

\tableofcontents


\section{Introduction}
Induced representations play very important role in representation theory of Lie algebras as a main tool to construct new irreducible modules and to approach the problem of its classification.  Parabolic induction associated with a parabolic subalgebra of an  affine Kac-Moody algebra defines a functor from the category of modules over the Levi factor of the parabolic subalgebra to the category of induced modules for the affine Lie algebra. 
 In this paper we consider parabolic induction when a parabolic subalgebra has infinite-dimensional  Levi factor.
Study of such  induced modules goes back to \cite{JK}, \cite{F1}, \cite{F2}, \cite{BBFK}, \cite{FK1}, \cite{FK2}, \cite{FK3} and references therein. Their structure   is understood in the case when the central element acts as a nonzero scalar for some categories of modules over the Levi factor of a parabolic subalgebra.   In particular, in a recent paper \cite{GKMOS}, it was shown that parabolic induction preserves irreducibility for tensor weight  inducing modules with nonzero central charge. These modules are constructed as tensor product of a weight module over the "affine part" of the Levi factor and a weight module over its "Heisenberg part". Previously, particular cases were considered in \cite{BBFK}, 
 \cite{FK1}, \cite{FK2}, \cite{FK3}. Also, free field realizations of induced modules were obtained in \cite{JK}, \cite{C}, \cite{CF} and 
 \cite{FKS}.   In particular, the last paper provided a uniform construction of  free field realizations  for an arbitrary affine Kac--Moody algebra.

The purpose of  current paper is to extend the results of \cite{GKMOS} to the case when the inducing representations are
 Whittaker modules or mixed tensor modules of weight and Whittaker representations.

 Whittaker modules for finite-dimensional semisimple Lie algebras were introduced in \cite{Ko} and 
 corresponding parabolic induction (for parabolic subalgebras with finite-dimensional Levi factor)
 from irreducible Whittaker modules was considered in \cite{Mc} and \cite{MS}.  Whittaker modules for Affine Lie algebras were studied in \cite{ALZ} and \cite{GZ}. Unlike in the classical case the structure of Whittaker modules 
 is well understood only for Affine $\mathfrak{sl}(2)$. For example,  the universal non-degenerate Whittaker modules  with noncritical level and non-degenerate Whittaker character are irreducible as in the classical case, while the Whittaker modules with  degenerate Whittaker character and critical level are connected with the vertex theory \cite{ALZ}.
 The structure problem  for Whittaker modules
  remains open in general.

   A different family of \emph{imaginary} Whittaker modules for non-twisted affine Lie algebras was considered in \cite{Chr} for the simplest case of a parabolic subalgebra with an  infinite-dimensional  Levi factor which a sum the Heisenberg and Cartan subalgebras. The inducing modules are Whittaker modules for the Heisenberg subalgebra. When the central charge is nonzero corresponding induced modules are irreducible. 
   
   We consider arbitrary parabolic subalgebra with infinite-dimensional  Levi factor and consider two induction functors: 
   from the category of  Whittaker modules over the Levi factor and from the category of mixed tensor modules. In both cases we show that the functors preserve irreducibility when the central charge is nonzero. This allows to construct 
   new families of irreducible representations for all Affine Lie algebras.   
   
   We briefly describe our main results. 
    Let $\widehat{\mathfrak g}$ be an  affine Kac-Moody algebra with a Cartan subalgebra $H$ and  $\widehat{\mathfrak p}$  a parabolic subalgebra $\widehat{\mathfrak g}$  with
 the Levi decomposition 
$$\widehat{\mathfrak p} =\widehat{\mathfrak l}\oplus  \widehat{\mathfrak u}_+,$$
  where $\widehat{\mathfrak l}$ is an infinite-dimensional  Levi factor.  For  a $\widehat{\mathfrak p}$-module $V$ such that $\widehat{\mathfrak u}_+V=0$ we consider 
 a $\widehat{\mathfrak p}$-induced  $\widehat{\mathfrak g}$-module
 $M_{\widehat{\mathfrak p}}(V)$. In this paper we determine the irreducibility of $M_{\widehat{\mathfrak p}}(V)$ for several families of irreducible $\widehat{\mathfrak p}$-modules $V$ with nonzero central charge, including \emph{imaginary Whittaker modules} and \emph{generalized imaginary Whittaker modules}. 
 
 Let $\widehat{\mathfrak l}_1$ be a Lie subalgebra of $\widehat{\mathfrak l}$ generated by the  root subspaces of 
 $\widehat{\mathfrak l}$ and by $d$, $H_{\mathfrak l}$ is its Cartan subalgebra,
 and $\mathfrak h_{\mathfrak l}$ is the finite dimensional part  of $H_{\mathfrak l}$. 
 Then 
 $\widehat{\mathfrak l}=\widehat{\mathfrak l}_1\oplus \mathfrak h_{\mathfrak l}^{\perp}$,   where $\mathfrak h_{\mathfrak l}^{\perp}\subset H$  is the orthogonal complement of $\mathfrak h_{\mathfrak l}$ with respect to the Killing form. Hence
$[\widehat{\mathfrak l}_1, \mathfrak h_{\mathfrak l}^{\perp}]=0$ and $(H\cap \widehat{\mathfrak l}_1)\cup \mathfrak h_{\mathfrak l}^{\perp}=H$.  

If  $\widehat{\mathfrak l}_1$ contains nonzero real root subspaces of 
 $\widehat{\mathfrak l}$ then 
assume that $\widehat{\mathfrak l}_1$ is the Lie subalgebra of $\widehat{\mathfrak l}$ generated by them   and by $d$, or equivalently
$\widehat{\mathfrak l}_1$ is the affinization (twisted or untwisted) of a semisimple Lie subalgebra $\mathfrak l_1$. In this case the Levi subalgebra $\widehat{\mathfrak l}$ does not necessarily contain the whole Heisenberg subalgebra. 

We summarize our first main result:

\medskip

\begin{theorem}\label{thm-I}   Let  $\widehat{\mathfrak l}=\widehat{\mathfrak l}_1\oplus \mathfrak h_{\mathfrak l}^{\perp}$,   $a\in\mathbb C\setminus \{0\}$, 
 $\lambda\in (\mathfrak h_{\mathfrak l}^{\perp})^*$, $\eta: U(\widehat{\mathfrak l}_1^+)\rightarrow \mathbb C$ a nonzero  homomorphism and $\widehat{W}_{\eta, a}$ an irreducible Whittaker  $\widehat{\mathfrak l}_1$-module    of type $\eta$, central charge $a$ and with the action of $\mathfrak h_{\mathfrak l}^{\perp}$ defined by $\lambda$.
  Then the induced generalized imaginary Whittaker
 $\widehat{\mathfrak g}$-module $M_{\widehat{\mathfrak p}}(\widehat{W}_{\eta, a})$ is irreducible. In particular, if 
 $\widehat{\mathfrak l}=G+ H$, where $G$ is the Heisenberg subalgebra of $\widehat{\mathfrak g}$,  $\eta: U(G_+)\rightarrow \mathbb C$  an algebra homomorphism such that $\eta |_{\widehat{\mathfrak g}_{n\delta}}\neq 0$ for infinitely many integer $n>0$ and $\widetilde{M}_{\eta, a}$  the Whittaker $G\oplus \mathbb C d$-module of type $\eta$ and central charge $a$, then the imaginary Whittaker
 $M_{\widehat{\mathfrak p}}(\widetilde{M}_{\eta, a})$ is irreducible.
\end{theorem}

\medskip

The first part of the theorem is proved in Theorem \ref{thm-main-2}, while the second part is Theorem \ref{thm-imag-Heis}.  The second part generalizes the results of \cite{Chr} to all affine Lie algebras. Theorem \ref{thm-I} produces new families of irreducible modules for affine Kac-Moody algebras from irreducible Whittaker modules. 
Moreover, in the case of an non-twisted Affine Lie algebra, this construction can be extended to the inducing modules $({W}_{\eta, a}\otimes E(\mu, a_1, \ldots, a_m))\otimes \mathbb C [d]$, where  ${W}_{\eta, a}$ is an irreducible Whittaker module of type $\eta$ and central charge $a$ over the central extension of the loop algebra and $E(\mu, a_1, \ldots, a_m)$ is the evaluation module (see Theorem \ref{thm-main-3}). 

\medskip

Next we consider induced modules from mixed tensor $\widehat{\mathfrak l}$-modules obtained as a tensor product of irreducible  modules for affine Lie subalgebras and  Whittaker modules for  Heisenberg subalgebras. The case when both components of the tensor module (over the affine subalgebra and the Heisenberg subalgebra) are weight modules was considered in  \cite{GKMOS}.

Let $\widehat{\mathfrak p} =\widehat{\mathfrak l}\oplus  \widehat{\mathfrak u}_+$ be a parabolic subalgebra of $\widehat{\mathfrak g} $ containing the whole Heisenberg subalgebra. Then
$$\widehat{\mathfrak l} =(\widetilde{\mathfrak l}_0+ \widetilde{G}(\widehat{\mathfrak l})^{\perp})\oplus \mathfrak h_{\mathfrak l}^{\perp},$$ where $\widetilde{\mathfrak l}_0=\widehat{\mathfrak l}_0\oplus \mathbb C d$ is an Affine subalgebra  of $\widehat{\mathfrak l}$ generated by the 
real root subspaces and extended by the derivation $d$ 
and $\widetilde{G}(\widehat{\mathfrak l})^{\perp}$ is the orthogonal complement of the Heisenberg subalgebra of $\widetilde{\mathfrak l}_0$ extended by the derivation. Consider the natural triangular decomposition 
$$\widetilde{G}(\widehat{\mathfrak l})^{\perp}=\widetilde{G}(\widehat{\mathfrak l})^{\perp}_-\oplus (\mathbb C c\oplus \mathbb C d)\oplus \widetilde{G}(\widehat{\mathfrak l})^{\perp}_+,$$
where $\widehat{\mathfrak g}_{k\delta}\cap \widetilde{G}(\widehat{\mathfrak l})^{\perp}\subset \widetilde{G}(\widehat{\mathfrak l})^{\perp}_{\pm} $ if and only if $k\in \mathbb Z_{\pm}$.

Our second main result is the following (see Theorem \ref{thm-main}).

\begin{theorem}\label{thm-II} 
\begin{itemize}
\item Let   $\lambda\in (\mathfrak h_{\mathfrak l}^{\perp})^*$, $a\in\mathbb C\setminus \{0\}$, $M$  an   $\widehat{\mathfrak l}_0$-module  with  central charge $a$.  Let $\eta: U(\widetilde{G}(\widehat{\mathfrak l})^{\perp}_+)\rightarrow \mathbb C$ be an algebra homomorphism, 
 and
  $S_{\eta, a}$ a Whittaker $ \widetilde{G}(\widehat{\mathfrak l})^{\perp}$-module of type $\eta$ and central charge $a$.
Consider a mixed tensor $\widehat{\mathfrak l}$-module $M\otimes S_{\eta, a}$ on which 
 $\mathfrak h_0^{\perp}$ acts via $\lambda$ and $d$ has a tensor product action. Then the induced 
 $\widehat{\mathfrak g}$-module $$M_{\widehat{\mathfrak p}}(M\otimes S_{\eta, a})=U(\widehat{\mathfrak g})\otimes_{U(\widehat{\mathfrak p})} (M\otimes S_{\eta, a})$$ is irreducible if and only if $M\otimes S_{\eta, a}$
is irreducible $\widehat{\mathfrak l}$-module.
\item Let $M$ be a weight  $\widetilde{\mathfrak l}_0$-module which is irreducible as $\widehat{\mathfrak l}_0$-module  with  central charge $a$.  If  $\eta |_{\widetilde{G}(\widehat{\mathfrak l})^{\perp}_+\cap \widehat{\mathfrak g}_{n\delta}}\neq 0$ for infinitely many integers $n>0$ then the induced 
 $\widehat{\mathfrak g}$-module $$M_{\widehat{\mathfrak p}}(M\otimes S_{\eta, a})=U(\widehat{\mathfrak g})\otimes_{U(\widehat{\mathfrak p})} (M\otimes S_{\eta, a})$$ is irreducible.
 \end{itemize}
\end{theorem}

We note that the first item of Theorem \ref{thm-II} is a very general statement which a priori does not impose any conditions on  $\widehat{\mathfrak l}_0$-module $M$. This allows to construct a large family of new irreducible representations for all Affine Lie algebras. In particular, Theorem \ref{thm-II} extends Theorem \ref{thm-I} to arbitrary parabolic subalgebras with infinite dimensional Levi factors.

\medskip

\section{Preliminaries}


\subsection{Affine Kac-Moody algebras}
Let $\widehat{\mathfrak g}$ be an  affine Kac-Moody algebra with a Cartan subalgebra $H$ and the root
system $\Delta$ of $\widehat{\mathfrak g}$. For a root $\alpha$ denote 
$$\widehat{\mathfrak g}_\alpha = \{ x \in \widehat{\mathfrak g} \, | \, [h, x] = \alpha(h) x
{\text{ for every }} h \in H\}.$$ Then $\widehat{\mathfrak g}$ has the  root
decomposition
$$
\widehat{\mathfrak g} = H \oplus (\oplus_{\alpha \in \Delta}
\widehat{\mathfrak g}_\alpha).
$$

 The set of all imaginary roots is $\Delta^{im}=\{k\delta |
k\in\mathbb Z\setminus\{0\}\}$.  

Let ${\mathfrak g}$ be  an underlined finite-dimensional simple Lie algebra of $\widehat{\mathfrak g}$ and ${\mathfrak h}\subset H$  a Cartan subalgebra of ${\mathfrak g}$. Then the real roots of $\widehat{\mathfrak g}$ 
 have the form $\phi+n\delta$ for some $\phi$  spanned  by the roots of ${\mathfrak g}$ and some $n\in \mathbb Z$.
We also have 
$H={\mathfrak h}\oplus \mathbb C c\oplus C d$,  where $c$ is a central element and $d\in H$ is  such that $\delta(d)=1$ (derivation of $\widehat{\mathfrak g}$).
 For any $x\in \widehat{\mathfrak g}_{\phi+n\delta}$ we have
 $[d,x]=nx$.

The Heisenberg subalgebra  $G\subset \widehat{\mathfrak g}$ is defined as follows: 
 $$G= \oplus_{k \in \mathbb Z \setminus \{0\}} \widehat{\mathfrak g}_{k\delta} \oplus \mathbb C c.$$ 
We have $G = G_- \oplus \mathbb C c \oplus G_+$, where
$G_{\pm} = \oplus_{k>0} \widehat{\mathfrak g}_{\pm k \delta}$.

Let  $\widehat{\mathfrak p}$  be a parabolic subalgebra $\widehat{\mathfrak g}$  with
 the Levi decomposition 
$$\widehat{\mathfrak p} =\widehat{\mathfrak l}\oplus  \widehat{\mathfrak u}_+,$$
  where $\widehat{\mathfrak l}$ is the Levi factor and $\widehat{\mathfrak u}_+$ is the radical.    
  
   Denote by $ \widehat{\mathfrak u}_-$ the opposite radical, that is if $\widehat{\mathfrak g}_{\alpha}\subset \widehat{\mathfrak u}_+$ then $\widehat{\mathfrak g}_{-\alpha}\subset \widehat{\mathfrak u}_-$. Then we have  $\widehat{\mathfrak g}= \widehat{\mathfrak p} \oplus \widehat{\mathfrak u}_-$.

 Let  $V$ an $\widehat{\mathfrak l}$-module.  Then we have   
   a $\widehat{\mathfrak p}$-module structure on $V$ by setting $\widehat{\mathfrak u}_+V=0$. 
 Define an induced  $\widehat{\mathfrak g}$-module
 $$M_{\widehat{\mathfrak p}}(V)=U(\widehat{\mathfrak g})\otimes_{U(\widehat{\mathfrak p})}V,$$
 which is isomorphic to 
 $U(\widehat{\mathfrak u}_-)\otimes V$ as a vector space, where $\widehat{\mathfrak u}_-$ is the opposite radical of $\widehat{\mathfrak p}$.

  If $V$ is an irreducible $\widehat{\mathfrak g}$-module then the action of $c$ on $V$ is scalar and its value is called the \emph{central charge} of $V$.
  
   A $\widehat{\mathfrak g}$-module $V$ is called \emph{weight} (with respect to the Cartan subalgebra $H$) if 
  
  \begin{equation*}
M=\bigoplus_{\lambda \in{H^*}}M_{\lambda},
\end{equation*}

where $M_{\lambda}= \{v \in M | \,   hv=\lambda(h)v, \, \forall h \in H^*\}$.
  
  If $\mathfrak a\subset \widehat{\mathfrak g}$ is a Lie subalgebra then weight structure on $\mathfrak a$-modules is determined  with respect to $\mathfrak a\cap H$.
   
  \medskip

\subsection{Whittaker modules for Affine Lie algebras}
  
  Consider the triangular  decomposition $\widehat{\mathfrak g}=\widehat{\mathfrak g}_-\oplus H\oplus \widehat{\mathfrak g}_+$, where $\widehat{\mathfrak g}_{\pm}=\sum_{\alpha\in \Delta_{\pm}}\widehat{\mathfrak g}_{\alpha}$ and $\Delta_{\pm}$ are positive and negative roots. 
  
  Let $\eta: U(\widehat{\mathfrak g}_+)\rightarrow \mathbb C$ be an algebra homomorphism with $\eta(\widehat{\mathfrak g}_+)\neq 0$.  Then  $\eta ([\widehat{\mathfrak g}_+, \widehat{\mathfrak g}_+])=0$. In particular, $\eta (\widehat{\mathfrak g}_{k\delta})=0$ for all $k>0$. 
  
  A $\widehat{\mathfrak g}$-module $V$ is a \emph{Whittaker module of type $\eta$} if $V$ is generated   by a \emph{Whittaker element} $v$ such that
$xv=\eta(x)v$ for any $x\in U(\widehat{\mathfrak g}_+)$.

For $a\in\mathbb C$ consider a $1$-dimensional $\widehat{\mathfrak g}_+\oplus \mathbb C c$-module $\mathbb C v$ such that $v$ is a Whittaker element of type $\eta$ and $cv=a v$. Define the following   \emph{universal}  Whittaker module of type $\eta$ and central charge $a$:
$$V_{\eta, a}=U(\widehat{\mathfrak g})\otimes_{U(\widehat{\mathfrak g}_+\oplus \mathbb C c)}\mathbb C v.$$
  Any  quotient  of 
  $V_{\eta, a}$ is a Whittaker module of type $\eta$ \cite[Lemma 5.2]{GZ}.
  
  Denote by $\mathfrak G$ the Affine Lie subalgebra of $\widehat{\mathfrak g}$ obtained by removing the derivation $d$: $\mathfrak G=\widehat{\mathfrak g}_-\oplus \mathfrak h\oplus \mathbb C c\oplus \widehat{\mathfrak g}_+$.
  Whittaker $\mathfrak G$-modules of type $\eta: U(\widehat{\mathfrak g}_+)\rightarrow \mathbb C$ are defined similarly as above. If ${W}_{\eta, a}$ is a  Whittaker $\mathfrak G$-modules of type $\eta$ and central charge $a$ then we get the induced Whittaker $\widehat{\mathfrak g}$-module of type $\eta$: 
  $$\widehat{W}_{\eta, a}=U(\widehat{\mathfrak g})\otimes_{U(\mathfrak G)} {W}_{\eta, a}\simeq \mathbb C[d]\otimes {W}_{\eta, a}.$$
  Note that the action of $d$ is free on $\widehat{W}_{\eta, a}$. On the other hand, for the same Whittaker function there could exist irreducible Whittaker  modules with free and with diagonalizable action of $d$ \cite{ALZ}.

\medskip

\section{Imaginary Whittaker modules for Affine Lie algebra}

\subsection{Whittaker modules for Heisenberg subalgebra}

Let $\eta: U(G_+)\rightarrow \mathbb C$ be an algebra homomorphism with $\eta(G_+)\neq 0$. A $G$-module $V$ is a \emph{Whittaker module of type $\eta$} if $V$ is generated   by a \emph{Whittaker element} $v$ such that
$xv=\eta(x)v$ for any $x\in U(G_+)$ \cite{Chr}.

For $a\in\mathbb C$ consider a $1$-dimensional $G_+\oplus \mathbb C c$-module $\mathbb C v$ such that $v$ is a Whittaker element of type $\eta$ and $cv=av$. Define a Whittaker module of type $\eta$ as follows:
$$M_{\eta, a}=U(G)\otimes_{U(G_+\oplus \mathbb C c)}\mathbb C v.$$

Note that module $M_{\eta, a}$ is not $\mathbb Z$-graded. The following properties of $M_{\eta, a}$ were shown in \cite[Proposition 5, Proposition 6, Corollary 8]{Chr}:

\begin{proposition}\label{prop-whitt-notgraded}
\begin{itemize}
\item The module $M_{\eta, a}$ is $U(G_-)$-free;
\item  The module $M_{\eta, a}$ is irreducible if $a\neq 0$;
 In this case $M_{\eta, a}$ is the unique (up to  isomorphism) irreducible Whittaker module of type $\eta$ with central charge $a$.

\end{itemize}
\end{proposition}

Next we extend the  Heisenberg subalgebra  by the derivation $d$ and set 
 $\widetilde{G}=G\oplus \mathbb C d$.

A Whittaker $\widetilde{G}$-module of type $\eta$ and central charge $a\in\mathbb C$ is defined as follows:
$$\widetilde{M}_{\eta, a}=U(\widetilde{G})\otimes_{U({G}_+\oplus \mathbb C c)}\mathbb C v.$$

\begin{proposition}\label{prop-whitt-graded}
\begin{itemize}
\item The module $\widetilde{M}_{\eta, a}$ is $U({G}_-\oplus \mathbb C d)$-free;
\item The module $M_{\eta, a}$ is a $G$-submodule of $\widetilde{M}_{\eta, a}$;
\item  If $\eta |_{\widehat{\mathfrak g}_{n\delta}}\neq 0$ for infinitely many integer $n>0$ then the module $\widetilde{M}_{\eta, a}$ is irreducible. In this case $\widetilde{M}_{\eta, a}$ is the unique (up to  isomorphism) irreducible Whittaker module of type $\eta$ with central charge $a$.

\end{itemize}
\end{proposition}

\begin{proof}
Follows from 
Propositions 20, 21, 25 and  Corollary 28 of \cite{Chr}.
\end{proof}

\medskip

\subsection{Imaginary Whittaker modules}

Consider a parabolic subalgebra $\widehat{\mathfrak p} =\widehat{\mathfrak l}\oplus  \widehat{\mathfrak u}_+,$
where $\widehat{\mathfrak l}=G+ H=(G+\mathbb C d)\oplus \mathfrak h$ and $[G, \mathfrak h]=0$.

\medskip

 \begin{theorem}\label{thm-imag-Heis}  Let $a\in\mathbb C\setminus \{0\}$, 
$\lambda\in (\mathfrak h)^*$ and 
   $\eta: U(G_+)\rightarrow \mathbb C$  an algebra homomorphism such that $\eta |_{\widehat{\mathfrak g}_{n\delta}}\neq 0$ for infinitely many integer $n>0$. Then the induced 
 $\widehat{\mathfrak g}$-module $$M_{\widehat{\mathfrak p}}(\widetilde{M}_{\eta, a})=U(\widehat{\mathfrak g})\otimes_{U(\widehat{\mathfrak p})}\widetilde{M}_{\eta, a}$$ is irreducible, where the action of $\mathfrak h$ on $\widetilde{M}_{\eta, a}$ is defined via $\lambda$.
  \end{theorem}
 
 \begin{proof}
  For non-twisted Affine Lie algebras the statement is \cite[Theorem 50]{Chr}.  For an arbitrary Affine Lie algebra the result is a particular case of Theorem \ref{thm-main-2} shown below.

\end{proof}

\medskip

\section{Generalized Imaginary Whittaker modules}

Now we consider a general parabolic subalgebra $\widehat{\mathfrak p}=\widehat{\mathfrak l}\oplus \widehat{\mathfrak u}_+$ with infinite dimensional Levi factor $\widehat{\mathfrak l}$. 

\medskip

\subsection{Inducing from Whittaker modules}

Let   $\widehat{\mathfrak l}_1$  be  a Lie subalgebra of $\widehat{\mathfrak l}$
generated by the subspaces $\widehat{\mathfrak g}_{\alpha}\cap 
 \widehat{\mathfrak l}$, 
$\alpha\in \Delta$,  and $ d$.  Set  $G({\widehat{\mathfrak l}})=\widehat{\mathfrak l}\cap G$ for the Lie subalgebra of $\widehat{\mathfrak l}$ spanned by its imaginary root subspaces.
Denote by   
$H_{\mathfrak l}\subset H$ a Cartan subalgebra of $\widehat{\mathfrak l}_1$ and set $\mathfrak h_{\mathfrak l}=H_{\mathfrak l}\cap \mathfrak h$. Let 
$\mathfrak h_{\mathfrak l}^{\perp}$ be the orthogonal complement to $\mathfrak h_{\mathfrak l}$ in $\mathfrak h$ with respect to the Killing form. Then we have 
$\widehat{\mathfrak l}=\widehat{\mathfrak l}_1\oplus \mathfrak h_{\mathfrak l}^{\perp}$, $H= H_{\mathfrak l}\oplus \mathfrak h_{\mathfrak l}^{\perp}$ and 
 $[\widehat{\mathfrak l}_1, \mathfrak h_{\mathfrak l}^{\perp}]=0$.  For simplicity we assume now that If  $\widehat{\mathfrak l}_1$ contains nonzero real root subspaces of 
 $\widehat{\mathfrak l}$ then 
$\widehat{\mathfrak l}_1$ is  generated by them   and by $d$. This means that 
$\widehat{\mathfrak l}_1$ is the affinization (twisted or untwisted) of a semisimple or abelian Lie subalgebra $\mathfrak l_1$. Under this assumption we have  $G({\widehat{\mathfrak l}})\subset 
[{\widehat{\mathfrak l}}, {\widehat{\mathfrak l}}]=[{\widehat{\mathfrak l}}_1, {\widehat{\mathfrak l}}_1]$.  Consider the standard triangular decomposition of $\widehat{\mathfrak l}_1$: 
 $\widehat{\mathfrak l}_1=\widehat{\mathfrak l}_1^-\oplus H_{\mathfrak l}\oplus \widehat{\mathfrak l}_1^+$, where 
 $\widehat{\mathfrak l}_1^{\pm}=\sum_{\alpha\in \Delta_{\pm}}\widehat{\mathfrak g}_{\alpha}\cap \widehat{\mathfrak l}_1$, and set $G({\widehat{\mathfrak l}})_{\pm}=G({\widehat{\mathfrak l}})\cap \widehat{\mathfrak l}_1^{\pm}$.
Let ${G}(\widehat{\mathfrak l})^{\perp}$ be the orthogonal complement of $G({\widehat{\mathfrak l}})$ (with respect to the Killing form), that is $G=G(\widehat{\mathfrak l})+G(\widehat{\mathfrak l})^{\perp}$, $[G(\widehat{\mathfrak l})^{\perp}, \widehat{\mathfrak l}_1]=0$ and  $G(\widehat{\mathfrak l})\cap G(\widehat{\mathfrak l})^{\perp}=\mathbb C c$.  

We have a natural triangular decomposition 
$${G}(\widehat{\mathfrak l})^{\perp}={G}(\widehat{\mathfrak l})^{\perp}_-\oplus \mathbb C c\oplus {G}(\widehat{\mathfrak l})^{\perp}_+,$$
where $\widehat{\mathfrak g}_{k\delta}\cap {G}(\widehat{\mathfrak l})^{\perp}\subset \widetilde{G}(\widehat{\mathfrak l})^{\perp}_{\pm} $ if and only if $k\in \mathbb Z_{\pm}$. Under our assumption we have 
${G}(\widehat{\mathfrak l})^{\perp}_+\subset \widehat{\mathfrak u}_+$.

The general case of an arbitrary reductive $\mathfrak l_1$ will be discussed in the next section.
If $\widehat{\mathfrak l}_1$ has no nonzero real root subspaces then without loss of generality  we assume that 
it contains the whole Heisenberg subalgebra, that is $G({\widehat{\mathfrak l}})=G$. 

For a nonzero homomorphism $\eta: U(\widehat{\mathfrak l}_1^+)\rightarrow \mathbb C$ and $a\in \mathbb C \setminus \{0\}$
 consider an irreducible Whittaker  $\widehat{\mathfrak l}_1$-module $\widehat{W}_{\eta, a}$   of type $\eta$  and central charge $a$.

 \begin{lemma}\label{lem-whitt}
 If $a\neq 0$ then $G({\widehat{\mathfrak l}})_-$ is torsion free on $\widehat{W}_{\eta, a}$.
 
 \end{lemma}

 \begin{proof}
 Let $v_{\eta}$ be a Whittaker vector in $\widehat{W}_{\eta, a}$. Set $\widehat{\mathfrak l}_{k\delta}=\widehat{\mathfrak g}_{k\delta}\cap \widehat{\mathfrak l}$. 
 Then $\widehat{\mathfrak l}_{k\delta}v_{\eta}=0$  
for any $k>0$ as $\widehat{\mathfrak l}_{k\delta}\subset  [\widehat{\mathfrak l}_1^+, \widehat{\mathfrak l}_1^+]$ and
 $\eta([\widehat{\mathfrak l}_1^+, \widehat{\mathfrak l}_1^+])=0$. Suppose that  for some nonzero $x\in \widehat{\mathfrak l}_{-k\delta}$ with $k>0$ and some $u\in U(\widehat{\mathfrak l}_1^-)$ we have 
$xu v_{\eta}=0$ while $u v_{\eta}\neq 0$. Assume for simplicity that $\widehat{\mathfrak l}$ is non-twisted and  take any nonzero $\bar{x}\in \widehat{\mathfrak l}_{k\delta}$ such that $[\bar{x}, x]=c$. 
 If $\bar{x} u v_{\eta}=0$ then  
$$0=\bar{x} x u v_{\eta}=[\bar{x}, x] u v_{\eta}=a u v_{\eta},$$
which is a contradiction. 

 Assume $\bar{x} u v_{\eta}\neq 0$. We have  
$$\bar{x} uv_{\eta}=[\bar{x}, u]v_{\eta}+   u\bar{x} v_{\eta} = [\bar{x}, u]v_{\eta},$$
since  $\eta(\bar{x}) =0$. Applying $\bar{x}$ sufficiently many times we get that for some $m>0$, 
$ad(\bar{x})^m(u)\in U(\widehat{\mathfrak l}_1^+)$ and 
$$\bar{x}^{m+1} uv_{\eta}=\bar{x}ad(\bar{x})^m(u)v_{\eta}=[\bar{x}, ad(\bar{x})^m(u)]v_{\eta}=\eta([\bar{x}, ad(\bar{x})^m(u)])v_{\eta}=0.$$

Choose the smallest $k>0$ such that $\bar{x}^{k+1} uv_{\eta}=0$ but $\bar{x}^{k} uv_{\eta}\neq 0$. Now we have
$$0=\bar{x}^{k+1}x uv_{\eta}=[\bar{x}^{k+1}, x] uv_{\eta} + x \bar{x}^{k+1} uv_{\eta}=
(k+1)a \bar{x}^{k} uv_{\eta},
$$
giving a contradiction.

Similarly, one can extend the arguments above to an arbitrary nonzero element $y\in U(G({\widehat{\mathfrak l}})_-)$ and show that $y u v_{\eta}\neq 0$ if $u v_{\eta}\neq 0$ completing the proof. 
   The twisted case is treated analogously.
 
 \end{proof}

\medskip

\begin{theorem}\label{thm-main-2}
Let  $\widehat{\mathfrak l}=\widehat{\mathfrak l}_1\oplus \mathfrak h_{\mathfrak l}^{\perp}$ with semisimple or abelian ${\mathfrak l}_1$,  $\lambda\in (\mathfrak h_{\mathfrak l}^{\perp})^*$, $a\in\mathbb C\setminus \{0\}$, $\eta\neq 0$ and $\widehat{W}_{\eta, a}$ an irreducible Whittaker  $\widehat{\mathfrak l}_1$-module    of type $\eta$ and central charge $a$. Define the action of $\mathfrak h_{\mathfrak l}^{\perp}$ on $\widehat{W}_{\eta, a}$ by $\lambda$ making it a $\widehat{\mathfrak p}$-module with $\widehat{\mathfrak u}_+\cdot \widehat{W}_{\eta, a}=0$.
  Then the induced 
 $\widehat{\mathfrak g}$-module $$M_{\widehat{\mathfrak p}}(\widehat{W}_{\eta, a})=U(\widehat{\mathfrak g})\otimes_{U(\widehat{\mathfrak p})} \widehat{W}_{\eta, a}$$ is irreducible.

\end{theorem}

\begin{proof}

Take any nonzero element of $v\in M_{\widehat{\mathfrak p}}(\widehat{W}_{\eta, a})$. By the PBW theorem the lement  $v$ can be written in the following form:
$$v=\sum_{i} u_i w_i,$$ where 
 $u_i\in U(\widehat{\mathfrak{u}}_-)$, $w_i\in \widehat{W}_{\eta, a}$.
 We assume that $u_i$ are linearly independent and $w_i\neq 0$ for all $i$. We can also assume that $v$ is $\mathfrak h_0^{\perp}$-weight element. 
  Then for each $i$, $u_i\in U(\widehat{\mathfrak{g}})_{-\phi_i + k_i \delta}$ for some  $\phi_i$ generated by the roots of $\mathfrak g$ and some integer $k_i$. Decompose each $\phi_i$ into a sum of simple roots. Note that the total number of simple roots of $\mathfrak g$ in $\phi_i$ which are not roots of $\mathfrak l$ is the same for all $i$ due to the fact that $v$ is an $\mathfrak h_0^{\perp}$-weight element. Denote this number by 
 $\tau (v)$ .   Also, denote by $\tau_i(v)$ the total number of simple roots of $\mathfrak l$ in $\phi_i$. Note that $\tau_i(v)$ can be different from $\tau_j(v)$ if $i\neq j$. We proceed by induction on $\tau(v)$.
  Suppose  that $\tau(v)=1$ and $\beta $ is a unique root of $\mathfrak{g}$ in all $\phi_i$. Then for each $i$,  $u_i \in \widehat{\mathfrak{g}}_{-\phi_i +k_i\delta}$ and $\phi_i =\beta +\alpha_i$, where  $\alpha_{i}$ decomposes into a sum of simple roots of $\mathfrak l$. Consider such indices $i$ for which 
  $\tau_i(v)$ is the least possible and among those the ones for which $k_i$ is the least possible. Fix an index $j$ satisfying all these conditions. Without loss of generality we may assume that $j=1$. For a sufficiently large $N$
($N>>k_i$ for all $i$)  choose a nonzero
  $u_N \in \widehat{\mathfrak{g}}_{\beta + \alpha_1 +(-N-k_1)\delta}$. Then we have 
 $$u_N \cdot v=\sum_i u_N u_i w_i=\sum_i [u_N, u_i] w_i=x_{-N}w_1+\sum_{j\geq 2} y_N^j w_j,$$
 where  $x_{-N}\in \widehat{\mathfrak{g}}_{-N\delta}$ is nonzero and  $y_N^{j}\in U(\widehat{\mathfrak{l}})$ for all $j$. Rewriting the sum if necessary we can assume that all summands are linearly independent.
 In the twisted case we restrict ourselves to those $N$ for which $\beta + \alpha_1 +(-N-k_1)\delta$ is a root. Note  $x_{N}w_i= 0$  for all $i$ if $N$ is sufficiently large. Moreover, for such $N$, we have 
 $x_{-N}w_1\neq 0$. Indeed, write $x_{-N}=x_{-N}^{\mathfrak{l}}+x_{-N}^{\mathfrak{u}}$, where $x_{-N}^{\mathfrak{l}}\in \widehat{\mathfrak{l}}_{-N\delta}$ and $x_{-N}^{\mathfrak{u}}\in {G}(\widehat{\mathfrak l})^{\perp}_-$. Then $x_{-N}^{\mathfrak{l}}w_1\neq 0$
  by Lemma \ref{lem-whitt} and $x_{-N}^{\mathfrak{u}}w_1\neq 0$ as the central charge is nonzero. Choose 
  $x_{N}^{\mathfrak{u}}\in {G}(\widehat{\mathfrak l})^{\perp}_+$ such that $[x_{N}^{\mathfrak{u}}, x_{-N}^{\mathfrak{u}}]=c$.
  Since 
  $[{G}(\widehat{\mathfrak l})^{\perp}_+, x_{-N}^{\mathfrak{l}}]=0$ we have 
  $$x_{N}^{\mathfrak{u}}x_{-N}w_1=x_{N}^{\mathfrak{u}}(x_{-N}^{\mathfrak{l}}+x_{-N}^{\mathfrak{u}})w_1=aw_1.$$
 Hence, $x_{-N}w_1\neq 0$.
  We claim now that
   $u_N v\neq 0$. Indeed, suppose $u_N v= 0$ for all  admissible $N$.   Then choose a nonzero $x_{N}\in \widehat{\mathfrak{g}}_{N\delta}$ such that $[x_{N}, x_{-N}]=c$. 
   Suppose that $[x_N, y_{N}^j]w_j= 0$ for all $j\geq 2$. Then $0=x_{N} u_N v=x_{N}x_{-N} w_1=a w_1$, which is a contradiction. Let $[x_N, y_{N}^j]w_j\neq 0$ for some $j\geq 2$ and assume $j=2$ for simplicity. Then
   $$0=x_N x_{-2N}w_1+\sum_{j\geq 2} x_N  y_{2N}^j w_j=\sum_{j\geq 2} [x_N,  y_{2N}^j] w_j=\sum_{j\geq 2} b_j  y_{N}^j w_j,$$ for some $b_j\in \mathbb C$ and $b_2\neq 0$. Hence $y_{N}^j w_j$, $j\geq 2$ are linearly dependent,  which is a contradiction. 
   
Suppose  now $\tau(v)>1$. Then the same argument as in the proof of  \cite[Lemma 5.3]{BBFK} shows that
there exists $u\in U(\widehat{\mathfrak u}_+)$  such that $u\cdot v\neq 0$ and $\tau(u \cdot v)< \tau(v)$. Then the proof is completed by induction. 
\end{proof}

\begin{remark} One can extend Lemma \ref{lem-whitt} and Theorem \ref{thm-main-2} to the case when the whole Heisenberg subalgebra $G$ is contained in $\widehat{\mathfrak l}$, that is ${G}(\widehat{\mathfrak l})^{\perp}\subset \widehat{\mathfrak l}$. This will follow from Theorem \ref{thm-main} as a particular case. 
\end{remark}

\medskip

\subsection{Inducing from Whittaker and evaluation modules}

In this section we assume that $\widehat{\mathfrak g}$ is a non-twisted Affine Lie algebra. 
We will extend the results of the previous section to the family of smooth inducing $\widehat{\mathfrak l}$-modules 
 which are the tensor products of Whittaker and evaluation modules \cite{GZ}.

Let ${\mathfrak l}_1\subset \mathfrak g$ be a semisimple Lie subalgebra, 
 $\widehat{\mathfrak l}=\widehat{\mathfrak l}_1\oplus \mathfrak h_{\mathfrak l}^{\perp}$, $\widehat{\mathfrak l}_1= \mathfrak l_1\otimes \mathbb C [t, t^{-1}]\oplus \mathbb C c \oplus \mathbb C d$ and  $\widehat{\mathfrak l}_2= \mathfrak l_1\otimes \mathbb C [t, t^{-1}]\oplus \mathbb C c$. 
For a nonzero homomorphism $\eta: U(\widehat{\mathfrak l}_1^+)\rightarrow \mathbb C$ and $a\in \mathbb C \setminus \{0\}$ let
 ${W}_{\eta, a}$  be an irreducible Whittaker  $\widehat{\mathfrak l}_2$-module    of type $\eta$  and central charge $a$. 
 
 For every positive integer $m$ consider $\mu=(\mu_1, \ldots, \mu_m)\in (\mathfrak {h}_{\mathfrak l}^*)^m\setminus \{0\}$ with all dominant integral $\mu_i$ and  a sequence $a_1, \ldots, a_m$ with all  nonzero  and pairwise distinct complex entries. Consider
 the evaluation $\widehat{\mathfrak l}_2$-module $E(\mu, a_1, \ldots, a_m)$ which is a tensor product of finite-dimensional irreducible $\mathfrak l_2$-modules with highest weights $\mu_1, \ldots, \mu_m$ and the action is defined as follows: 
 
$$(x\otimes t^n)(v_1\otimes \ldots \otimes v_m)=\sum_{i=1}^m a_i^n(v_1\otimes \ldots \otimes xv_i \otimes \ldots \otimes v_m).$$ 
 
Taking the tensor product ${W}_{\eta, a}\otimes E(\mu, a_1, \ldots, a_m)$ we get an $\widehat{\mathfrak l}_2$-module which  induces the following $\widehat{\mathfrak l}_1$-module \cite{GZ}:
 $$({W}_{\eta, a}\otimes E(\mu, a_1, \ldots, a_m))[d]=U(\widehat{\mathfrak l}_1)\otimes_{U(\widehat{\mathfrak l}_2)} ({W}_{\eta, a}\otimes E(\mu, a_1, \ldots, a_m)).$$ 

Theorem \ref{thm-main-2} can be easily extended to the inducing modules $({W}_{\eta, a}\otimes E(\mu, a_1, \ldots, a_m))[d]$. 

\medskip

\begin{theorem}\label{thm-main-3}
Let   $\lambda\in (\mathfrak h_{\mathfrak l}^{\perp})^*$, $a\in\mathbb C\setminus \{0\}$, $\eta\neq 0$,  $\mu=(\mu_1, \ldots, \mu_m)\in (\mathfrak {h}_{\mathfrak l}^*)^m\setminus \{0\}$ with all dominant integral $\mu_i$ and $a_1, \ldots, a_m$   pairwise distinct nonzero  numbers. Let  ${W}_{\eta, a}$  be an irreducible Whittaker  $\widehat{\mathfrak l}_2$-module    of type $\eta$  and central charge $a$. 
Define the action of $\mathfrak h_{\mathfrak l}^{\perp}$ on $({W}_{\eta, a}\otimes E(\mu, a_1, \ldots, a_m))[d]$ by $\lambda$ making it a $\widehat{\mathfrak p}$-module with trivial action of $\widehat{\mathfrak u}_+$.
  Then the induced 
 $\widehat{\mathfrak g}$-module $$M_{\widehat{\mathfrak p}}(({W}_{\eta, a}\otimes E(\mu, a_1, \ldots, a_m))[d])=U(\widehat{\mathfrak g})\otimes_{U(\widehat{\mathfrak p})} ({W}_{\eta, a}\otimes E(\mu, a_1, \ldots, a_m))[d]$$ is irreducible.

\end{theorem}

\begin{proof}
Note that $({W}_{\eta, a}\otimes E(\mu, a_1, \ldots, a_m))[d]$ is irreducible $\widehat{\mathfrak l}_1$-module by \cite[Corollary 3.5]{GZ}. The proof is analogous to the proof of Theorem \ref{thm-main-2}. We leave the details to the reader. 

\end{proof}

\medskip

\subsection{Inducing from mixed tensor modules}

 From now on we assume that the Levi factor $\widehat{\mathfrak l}=\widehat{\mathfrak l}_1+ \mathfrak h_{\mathfrak l}^{\perp}$ of the parabolic subalgebra $\widehat{\mathfrak p}=\widehat{\mathfrak l}\oplus \widehat{\mathfrak u}_+$
 contains the whole Heisenberg subalgebra $G$.

Denote by $\widehat{\mathfrak l}_0$ the Lie subalgebra of $\widehat{\mathfrak l}$ generated by all its real root subspaces. Let $G({\widehat{\mathfrak l}})$ be the Lie subalgebra of $\widehat{\mathfrak l}_0$ spanned by its imaginary root subspaces,  $G(\widehat{\mathfrak l})^{\perp}$
  the orthogonal complement 
    of  $G(\widehat{\mathfrak l})$ in $G$: $G=G(\widehat{\mathfrak l})+G(\widehat{\mathfrak l})^{\perp}$, $[G(\widehat{\mathfrak l})^{\perp}, \widehat{\mathfrak l}_0]=0$ and  $ \widehat{\mathfrak l}_0\cap G(\widehat{\mathfrak l})^{\perp}=\mathbb C c$.

We have $$\widehat{\mathfrak l} =\widehat{\mathfrak l}_0+ G(\widehat{\mathfrak l})^{\perp}+ \mathfrak h_{\mathfrak l}^{\perp} + \mathbb C d=(\widetilde{\mathfrak l}_0+ \widetilde{G}(\widehat{\mathfrak l})^{\perp})\oplus \mathfrak h_{\mathfrak l}^{\perp},$$
where $\widetilde{\mathfrak l}_0=\widehat{\mathfrak l}_0 \oplus \mathbb C d$ and $ \widetilde{G}(\widehat{\mathfrak l})^{\perp}=G(\widehat{\mathfrak l})^{\perp}\oplus \mathbb C d$. The Lie algebras $G(\widehat{\mathfrak l})^{\perp}$ and  $\widetilde{G}(\widehat{\mathfrak l})^{\perp}$ inherit the following triangular decompositions 
$$G(\widehat{\mathfrak l})^{\perp}=G(\widehat{\mathfrak l})^{\perp}_{-}\oplus \mathbb C c\oplus G(\widehat{\mathfrak l})^{\perp}_{+}, \,\,\, \widetilde{G}(\widehat{\mathfrak l})^{\perp}=\widetilde{G}(\widehat{\mathfrak l})^{\perp}_{-}\oplus (\mathbb C c\oplus \mathbb C d)\oplus \widetilde{G}(\widehat{\mathfrak l})^{\perp}_+$$ from the corresponding decomposition of $G$. Here $G(\widehat{\mathfrak l})^{\perp}_{\pm}= \widetilde{G}(\widehat{\mathfrak l})^{\perp}_{\pm}=G(\widehat{\mathfrak l})^{\perp}\cap G_{\pm}.$

The following proposition shows how to construct irreducible $\widehat{\mathfrak l}$-modules.

\medskip

\begin{proposition}\label{prop-constr-irr}

\begin{itemize}
\item Let $\lambda\in (\mathfrak h_{\mathfrak l}^{\perp})^*$,  $M$ and $S$ are irreducible modules over $\widehat{\mathfrak l}_0$ and  $G(\widehat{\mathfrak l})^{\perp}$ respectively with the same central charge $a\in\mathbb C$. Then $M\otimes S$ is 
an irreducible $(\widehat{\mathfrak l}_0+ G(\widehat{\mathfrak l})^{\perp}+ \mathfrak h_{\mathfrak l}^{\perp})$-module with central charge $a$ and with the action of $\mathfrak h_{\mathfrak l}^{\perp}$ defined by $\lambda$: $h(v\otimes w)=\lambda(h)(v\otimes w)$ for all $h\in \mathfrak h_{\mathfrak l}^{\perp}$, $v\in M$, $w\in S$.
\item Let $V$ be an irreducible  $(\widehat{\mathfrak l}_0+ G(\widehat{\mathfrak l})^{\perp}+ \mathfrak h_{\mathfrak l}^{\perp})$-module with central charge $a$ and with the action of $\mathfrak h_{\mathfrak l}^{\perp}$ defined by some $\lambda\in (\mathfrak h_{\mathfrak l}^{\perp})^*$. Suppose $S$ is an irreducible 
 $G(\widehat{\mathfrak l})^{\perp}$-submodule of $V$. Then $V\simeq M\otimes S$ for some irreducible  $\widehat{\mathfrak l}_0$-module $M$ with central charge $a$. 
\item  Let $\lambda\in (\mathfrak h_{\mathfrak l}^{\perp})^*$,  $M$  a weight   $\widetilde{\mathfrak l}_0$-module which irreducible as $\widehat{\mathfrak l}_0$-module  with  central charge $a\in\mathbb C\setminus \{0\}$,  $S$ an irreducible Whittaker $ \widetilde{G}(\widehat{\mathfrak l})^{\perp}$-module of type $\eta$ with $\eta |_{{G}(\widehat{\mathfrak l})^{\perp}\cap \widehat{\mathfrak g}_{n\delta}}\neq 0$ for infinitely many integer $n>0$ and the same central charge. Then $M\otimes S$ is 
an irreducible $\widehat{\mathfrak l}$-module with central charge $a$, with the action of $\mathfrak h_{\mathfrak l}^{\perp}$ defined by $\lambda$ and with the tensor product action of $d$.

\end{itemize}
\end{proposition}

\begin{proof}
If $M$ and $S$ have central charge $a$ then they are irreducible modules over $U(\widehat{\mathfrak l}_0)/(c-a)$ and  $U(G(\widehat{\mathfrak l})^{\perp})/(c-a)$ respectively. Hence, $M\otimes S$ is an irreducible  $U(\widehat{\mathfrak l}_0 + G(\widehat{\mathfrak l})^{\perp})/(c-a)$-module and the first statement is clear. 
The second statement follows from \cite[Lemma 2.2]{LX}.

For the third statement, suppose $M\otimes S$ contains a nonzero  $\widehat{\mathfrak l}$-submodule $D$ and
consider any nonzero element $v\in D$. Then
$$v=\sum_{i\in I} v_i\otimes w_i,$$
for some finite set $I$, where $v_i\in M$ and $w_i\in S$. We can assume that $w_i=d^{k_i}z_i$, $i\in I$, where $z_i\in U({G}(\widehat{\mathfrak l})^{\perp}_-)$, $k_i\geq 0$, and the elements $v_i\otimes w_i$ are linearly independent.  If $i\in I$ and $n>0$ then for any nonzero
$x\in \widetilde{G}(\widehat{\mathfrak l})^{\perp}\cap \widehat{\mathfrak g}_{n\delta}$ such that $[x, z_i]=0$,  we have
$$(x-\eta(x))(d^{k_i}z_i)=\sum_{j=1}^{k_i-1}\eta(x)\lambda_j d^j z_i$$
for some integers $\lambda_j$.  Note that the degree of $d$ in the right hand side is smaller that in the original element. 
Since  the restrictions of $\eta$ on ${G}(\widehat{\mathfrak l})^{\perp}\cap{\widehat{\mathfrak g}_{n\delta}}$ are nonzero for infinitely many $n>0$, there exists $u\in U(\widetilde{G}(\widehat{\mathfrak l})^{\perp}_+)$ such that 
$$v'=u\cdot v=\sum_{i\in I'} v_i'\otimes w_i',$$
for some new set of indices $I'$, $v_i'\in M$ and  $w_i'\in M_{\eta, a}$. Again we assume that all $v_i'\otimes w_i'$ 
are linearly independent.
As the Whittaker ${G}(\widehat{\mathfrak l})^{\perp}$-module $M_{\eta, a}$ is irreducible, we find 
$u'\in U(\widetilde{G}(\widehat{\mathfrak l})^{\perp}_+)$ such that 
$$u'\cdot v'=\sum_{i\in I''} v_i''\otimes \bar{1},$$
where $\bar{1}$ is the generator of $S$,  for some  set of indices $I''$ and some  linearly independent elements $v_i''\in M$. Then $u'\cdot v'\in D$ is nonzero. We see that $D$ contains a nonzero element $m\otimes \bar{1}$
with $m\in M$.
Then $U(\widehat{\mathfrak l}_0)(m)\otimes \bar{1}=M\otimes \bar{1}\subset D$  (recall that $M$  is irreducible as $\widehat{\mathfrak l}_0$-module). Finally we have
$$U(\widetilde{G}(\widehat{\mathfrak l})^{\perp})(M\otimes \bar{1})=M\otimes S\subset D,$$
implying that $D=M\otimes S$. Hence $M\otimes S$ is irreducible.

\end{proof}

\begin{remark}
Examples of weight   $\widetilde{\mathfrak l}_0$-module which are irreducible as $\widehat{\mathfrak l}_0$-module  include  irreducible Verma modules, imaginary Verma modules with nonzero central charge, and modules induced from irreducible diagonal $G$-modules with nonzero central charge, among the others.
\end{remark}

\begin{remark}
If $M$ is a weight   $\widetilde{\mathfrak l}_0$-module and $S$ is a weight $G(\widehat{\mathfrak l})^{\perp}$-module with the same central charge then $M\otimes S$ is a tensor $\widetilde{\mathfrak l}$-module \cite{FK3}, \cite{GKMOS}. Modules induced from weight tensor modules were studied in \cite{GKMOS}. Inspired by \cite{GKMOS} we will consider a different family of inducing modules which we call \emph{mixed tensor modules}. These are modules of the form 
$M\otimes S$ where  $M$ is an  $\widetilde{\mathfrak l}_0$-module and $S$ is a Whittaker  $G(\widehat{\mathfrak l})^{\perp}$-module with the same central charge. 

\end{remark}

Let $\eta: U(\widetilde{G}(\widehat{\mathfrak l})^{\perp}_+)\rightarrow \mathbb C$ be an algebra homomorphism, 
$a\in\mathbb C\setminus \{0\}$ and
  $S_{\eta, a}$ a Whittaker $ \widetilde{G}(\widehat{\mathfrak l})^{\perp}$-module of type $\eta$ and central charge $a$.

The following theorem  generalizes Theorem \ref{thm-imag-Heis} for arbitrary parabolic subalgebra and for mixed tensor modules. Note that we are not imposing any conditions on $\widehat{\mathfrak l}_0$-module $M$.

\medskip

\begin{theorem}\label{thm-main}
Let   $\lambda\in (\mathfrak h_{\mathfrak l}^{\perp})^*$, $a\in\mathbb C\setminus \{0\}$, $M$  an   $\widehat{\mathfrak l}_0$-module  with  central charge $a$. 
Consider a mixed tensor $\widehat{\mathfrak l}$-module $M\otimes S_{\eta, a}$ on which 
 $\mathfrak h_{\mathfrak l}^{\perp}$ acts via $\lambda$ and $d$ has a tensor product action. Then the induced 
 $\widehat{\mathfrak g}$-module $$M_{\widehat{\mathfrak p}}(M\otimes S_{\eta, a})=U(\widehat{\mathfrak g})\otimes_{U(\widehat{\mathfrak p})} (M\otimes S_{\eta, a})$$ is irreducible if and only if $M\otimes S_{\eta, a}$
is irreducible $\widehat{\mathfrak l}$-module.
\end{theorem}

\begin{proof}  
If $M\otimes S_{\eta, a}$ is not irreducible $\widehat{\mathfrak l}$-module then clearly $M_{\widehat{\mathfrak p}}(M\otimes S_{\eta, a})$ is not irreducible by the basic properties of induced modules. Conversely, 
assume that a mixed tensor $\widehat{\mathfrak l}$-module $M\otimes S_{\eta, a}$ is irreducible. 

 Let $v\in M_{\widehat{\mathfrak p}}(M\otimes S_{\eta, a})$ be a nonzero element. Then $v$ can be written as follows:
$$v=\sum_{i} u_i(v_i \otimes w_i),$$ for some
 $u_i\in U(\widehat{\mathfrak{u}}_-)$,  $v_i\in M$,      $w_i\in S_{\eta, a}$. We assume that the elements $u_i(v_i \otimes w_i)$ in the decomposition of $v$ are linearly independent.
 We can also assume that $v$ is $\mathfrak h_{\mathfrak l}^{\perp}$-weight element and hence  
 $u_i\in U(\widehat{\mathfrak{g}})_{-\varphi_i + k_i \delta}$ or each $i$, for some  $\varphi_i$ in the root lattice of $\mathfrak g$ and some integer $k_i$. Each $\varphi_i$ can be written as a sum of simple roots of $\mathfrak g$. Denote by  $\tau (\varphi_i)$ the total number of simple roots of $\mathfrak g$ in $\varphi_i$ which are not the roots of $\widehat{\mathfrak l}$.  Clearly $\tau (\varphi_i)=\tau (\varphi_j)$ for all $i,j$   since $v$ is a weight element.
 We denote this number by $\tau (v)$. Suppose that $\tau (v)=1$, that is $\varphi_i$ contains a unique simple root 
  which is not a root of $\widehat{\mathfrak l}$, and for all $i$ this root is the same. 
  As with the weight modules (see \cite{FK3}, \cite{GKMOS}) this the most difficult case which requires a special treatment. 
  
  Choose any index $i_0$ for which  
   $k_{i_0}$ is the largest (there could be several indices like that, in this case choose any of them),  
  and
  choose sufficiently large positive $N$. Let  $u\in \widehat{\mathfrak{g}}_{\varphi_{i_0} - (N+k_{i_0}) \delta}$ be a nonzero element. We have
 $$u\cdot v=\sum_{t\in T}y_{-N}(v_t\otimes w_t) +  \sum_{j\in J} y_{-N-s_j}(v_j\otimes w_j)+ \sum_{k\in K}  z_k(v_k\otimes w_k),$$
 where $T,J,K$ some sets of indices ($i_0\in T$), $s_j >0$, $y_{-N-s_j}\in \widehat{\mathfrak{g}}_{(-N-s_j)\delta}$, $z_k\in U(\widehat{\mathfrak l}_0)$.
 Moreover, each $y_{N-s_j}$ can be written as $y_{-N-s_j}^{(1)}+y_{-N-s_j}^{(2)}$, where  $y_{-N-s_j}^{(1)}\in \widehat{\mathfrak l}_0$ and  $y_{-N-s_j}^{(2)}\in \widetilde{G}(\widehat{\mathfrak l})^{\perp}_-$.   Hence we have
 $$u\cdot v=\sum_{t\in T}y_{-N}^{(1)}v_t\otimes w_t + \sum_{t\in T}v_t\otimes y_{-N}^{(2)}w_t + \sum_{j\in J} y_{-N-s_j}^{(1)}v_j\otimes w_j+ \sum_{j\in J} v_j\otimes y_{-N-s_j}^{(2)} w_j +
  \sum_{k\in K}  z_kv_k\otimes w_k.$$
 Since $u\cdot v\in M\otimes S_{\eta, a}$ we just need to show 
  that $u\cdot v\neq 0$ for some sufficiently large $N$. Then irreducibility follows. Note that $y_{-N}^{(2)}w_t\neq 0$  and $ y_{-N-s_j}^{(2)} w_j \neq 0$ for all $t\in T$,  $j\in J$ by Proposition \ref{prop-whitt-notgraded}.  Choose $y_{N}^{(2)}\in \widetilde{G}(\widehat{\mathfrak l})^{\perp}_+$ such that $[y_{N}^{(2)}, y_{-N}^{(2)}]=Nc$   and $y_{N}^{(2)} w_s=\eta(y_{N}^{(2)})w_s$ for all $s\in T\cup J$. 
   Then 
  $$y_{N}^{(2)}(u\cdot v)=\sum_{t\in T}y_{-N}^{(1)}v_t\otimes y_{N}^{(2)}w_t + \sum_{t\in T}v_t\otimes y_{N}^{(2)}y_{-N}^{(2)}w_t + \sum_{j\in J} y_{-N-s_j}^{(1)}v_j\otimes y_{N}^{(2)}w_j+ $$
  $$+\sum_{j\in J} v_j\otimes y_{N}^{(2)}y_{-N-s_j}^{(2)} w_j +
  \sum_{k\in K}  z_kv_k\otimes y_{N}^{(2)}w_k= \eta(y_{N}^{(2)})\sum_{t\in T}y_{-N}^{(1)}v_t\otimes w_t + \eta(y_{N}^{(2)})\sum_{t\in T}v_t\otimes y_{-N}^{(2)}w_t + $$  
 $$ +Na\sum_{t\in T}v_t\otimes w_t+ \eta(y_{N}^{(2)})\sum_{j\in J} y_{-N-s_j}^{(1)}v_j\otimes w_j   + \eta(y_{N}^{(2)})\sum_{j\in J} v_j\otimes y_{-N-s_j}^{(2)} w_j +$$
$$+
\eta(y_{N}^{(2)})  \sum_{k\in K}  z_kv_k\otimes w_k = \eta(y_{N}^{(2)}) (u\cdot v) + Na\sum_{t\in T}v_t\otimes w_t.$$
  
  Suppose that $u\cdot v=0$. Then $\sum_{t\in T}v_t\otimes w_t=0.$ But $u_i=u_j$ for all $i,j\in T$ implying that
  $\{u_i(v_i\otimes w_i), \, i\in T\}$ are linearly dependent, which is a contradiction. Hence $u\cdot v\neq 0$. 
  This comletes the proof in the case $\tau (v)=1$. If 
 $\tau(v)>1$ then by the argument  in the proof of  \cite[Lemma 5.3]{BBFK} one can easily find
  $u\in U(\widehat{\mathfrak u}_+)$ such that $u\cdot v\neq 0$ and $\tau(u\cdot v)< \tau(v)$. Then we complete the proof  by induction on $ \tau(v)$.

\end{proof} 
 
 In particular Theorem \ref{thm-main} is applied in the case of mixed tensor modules with weight $\widetilde{\mathfrak l}_0$-module  $M$.
 
 \begin{corollary}\label{cor-mixed} Let   $\lambda\in (\mathfrak h_{\mathfrak l}^{\perp})^*$, $a\in\mathbb C\setminus \{0\}$ and
 $M$  a weight  $\widetilde{\mathfrak l}_0$-module which is irreducible as $\widehat{\mathfrak l}_0$-module  with  central charge $a$.  If  $\eta |_{\widetilde{G}(\widehat{\mathfrak l})^{\perp}_+\cap \widehat{\mathfrak g}_{n\delta}}\neq 0$ for infinitely many integers $n>0$ then the induced 
 $\widehat{\mathfrak g}$-module $$M_{\widehat{\mathfrak p}}(M\otimes S_{\eta, a})=U(\widehat{\mathfrak g})\otimes_{U(\widehat{\mathfrak p})} (M\otimes S_{\eta, a})$$ is irreducible.
\end{corollary}

\begin{proof}
The module $S_{\eta, a}$ is irreducible by the assumption on $\eta$. Then $M\otimes S_{\eta, a}$ is irreducible by  Proposition \ref{prop-constr-irr}. It remains to apply Theorem \ref{thm-main}. 
\end{proof}

\begin{remark}
All results hold for more general parabolic subalgebras with infinite dimensional Levi factors that contain a proper subalgebra of the Heisenberg algebra. All proofs remain valid in this case and hence we obtain more 
families of irreducible modules for Affine Lie algebras.  
\end{remark}

\medskip

\section*{Acknowledgments}

M.C.\,C. is supported in part by the Fapesp (2019/24494-9).
V.\,F.\ is supported in part by the CNPq (304467/2017-0) and by the Fapesp (2018/23690-6). The authors are grateful to Lucas Calixto and Henrique Rocha for useful discussions and comments.



\end{document}